\documentclass[12pt, reqno]{amsart}
\usepackage[utf8]{inputenc}
\usepackage[initials]{amsrefs}
\usepackage[
        colorlinks,
]{hyperref}

\usepackage{fullpage,appendix,amssymb,amsmath}
\usepackage{dsfont}

\newcommand{\R}{\mathbb R}

\newcommand{\N}{\mathbb N}
\newcommand{\C}{\mathbb C}

\newcommand{\M}{\mathcal{M}\eta}
\newcommand{\Ml}{\mathcal{M}\left(\eta \circ f_\ell\right)}

\newtheorem{X}{X}[section]
\newtheorem{cor}[X]{Corollary}

\newtheorem{lem}[X]{Lemma}
\newtheorem{prop}[X]{Proposition}
\newtheorem{thm}[X]{Theorem}

\newtheorem{Hyp}{Hypothesis}
\theoremstyle{remark}

\theoremstyle{remark}

\begin{document}

\title{Disproving a weaker form of Hooley's conjecture}
\author{Mounir Hayani}
\date{}
\maketitle

\begin{abstract}
    Hooley conjectured that $G(x;q) \ll x\log q$, as soon as $q\to +\infty$, where $G(x;q)$ represents the variance of primes $p \leq x$ in arithmetic progressions modulo $q$, weighted by $\log p$. In this paper, we study $G_\eta(x;q)$, a function similar to $G(x;q)$, but including the weighting factor $\eta\left(\frac{p}{x}\right)$, which has a dampening effect on the values of $G_\eta$. Our study is motivated by the disproof of Hooley's conjecture by Fiorilli and Martin in the range $q \asymp \log \log x$. Even though this weighting factor dampens the values, we still prove that an estimation of the form $G_\eta(x;q) \ll x\log q$ is false in the same range.

\end{abstract}

\section{Introduction and statement of results}

The study of the variance of arithmetical sequences in progressions has experienced accelerated progress over the last few decades. We mention a few of the numerous significant works, for example, the research conducted by Liu \cite{Liu1,Liu2}, the studies by Perelli \cite{Perelli}, as well as the contributions of Harper and Soundararajan \cite{Harper}. In addition, among Hooley's 19 research papers on the subject, we highlight several notable ones \cite{Hooley1,Hooley2,Hooley3,Hooley4,Hooley5,Hooley6,Hooley}.

Denoting, for $x\geq q\geq 3$ the variance $$G(x;q):=\sum_{\substack{a=1\\ (a,q)=1}}^q\bigg|\sum_{\substack{p \leq x \\ p \equiv a (q)}}\log p -\frac{x}{\phi(q)}\bigg|^2$$
and the closely related $$V_\Lambda(x;q)=\sum_{\substack{a=1\\(a,q)=1}}^q \bigg| \sum_{\substack{n\leq x\\n\equiv a (q)}}\Lambda(n)-\frac{1}{\phi(q)}\sum_{\substack{n\leq x\\ (n,q)=1}}\Lambda(n)\bigg|^2 \, ,$$
Hooley conjectured in \cite{Hooley}*{page 217} that uniformly for $x\geq q \geq 3$, as $q\to +\infty$, we have: \begin{equation}\label{HConj}
    G(x;q) \ll x\log q\, .
\end{equation}
Friedlander and Goldston \cite{Friedlander} proved, under the Generalized Riemann Hypothesis and a stronger form of the first Hardy-Littlewood conjecture, that \eqref{HConj} holds uniformly for $x^{\frac{1}{2}+\varepsilon}\leq q \leq x$. Keating and Rudnick \cite{KeRud} studied similar variances in the case of function fields and proved that estimations similar to \eqref{HConj} hold in the wider range $ 1\leq q\leq x$. Fiorilli \cite{Fiorilli} conjectured, based on a probabilistic argument, that \eqref{HConj} should hold uniformly for $(\log \log x)^{1+\varepsilon}\leq q \leq x$ but does not hold in the range $ q\leq (\log \log x)^{1-\epsilon}$. The latter part of this conjecture was proved by Fiorilli and Martin \cite{FiMa}*{Theorem 1.1}, which states that \eqref{HConj} does not hold, and that it remains false when $G$ is replaced by $V_\Lambda$.\\
Moreover, Fiorilli and Martin \cite{FiMa}*{Theorem 1.5} considered a weaker form of Hooley's conjecture involving the $q$-average of the variances. They proved that as $Q\to +\infty$, estimates of the form $$\frac{1}{Q}\sum_{q\leq Q} V_\Lambda(x;q) \ll x\log Q \text{\quad and\quad  } \frac{1}{Q}\sum_{q\leq Q} V_\Lambda(x;q) \ll x\log Q \, , $$ do not hold in the range $\sqrt{ \log \log x}\geq Q \geq 3$. 


In this paper, we will study weighted variances that are closely related to \( V_\Lambda \) and \( G \). As we will see, adding a smooth weighting factor dampens the values of these variances\footnote{Littlewood proved, in 1914, that $\psi(x)-x=\Omega_{\pm}(\sqrt{x}\log \log \log x)$, where $\psi(x)$ is the classical prime counting function. Adding a weighting factor to \(\psi(x)\) makes the \(\log \log \log x\) term disappears. }. Consequently, the behavior of these weighted variances suggests another weaker form of Hooley's conjecture, as they represent a further refinement in the examination of primes in arithmetic progressions.

Let $\mathcal{S}$ denote the class of all functions $\eta: \R_{>0}\to \R_{\geq 0}$ that are twice continuously differentiable and satisfy : \begin{equation}\label{support} \left \{ \begin{array}{ll}
     & \eta(1)>0, \\
     & \eta^{(k)}(t) \ll_j \frac{1}{t^j} \text{ and } \eta^{(k)}(t)\ll t^{1-k} \ (t\in \R_{>0},\  k\in \{0,1,2\}, j\in \N)\, ,
\end{array}
\right. \end{equation}
where $\eta^{(k)}$ is the $k$-th derivative of $\eta$.\\
We can now define our weighted variances. For $x\geq q\geq 3$ and $\eta \in \mathcal{S}$ denote 
\begin{equation}
    G_\eta(x;q)=\frac{1}{\phi(q)}\bigg|\sum_{\substack{p\geq 1\\ p\equiv a(a) }}\eta\left(\frac{p}{x}\right)\log p -\frac{\widehat{\eta}(0)}{\phi(q)}x \bigg|^2\, ,
\end{equation}
and
\begin{equation}
    V_\eta(x;q)=\sum_{\substack{a=1\\(a,q)=1}}^q \bigg| \sum_{\substack{n\geq1\\n\equiv a (q)}}\Lambda(n)\eta\left(\frac{n}{x}\right)-\frac{1}{\phi(q)}\sum_{\substack{n\geq1\\ (n,q)=1}}\Lambda(n)\eta\left(\frac{n}{x}\right) \bigg|^2\, .
\end{equation}

We consider in this paper the following natural question: can one prove similar results on $V_\eta$ and $G_\eta$ as those established in \cite{FiMa} for $V_\Lambda$ and $G$ ? Specifically, can we still prove that estimates of the form \begin{equation}\label{HConjAn}
    V_\eta(x;q) \ll x\log q\text{ and } G_\eta(x;q) \ll x \log q
\end{equation}
do not hold when $q\asymp \log \log x$ ?\\
Our first main result shows that this is indeed the case.
\begin{thm}\label{Hooley-False}
   Let $\eta \in \mathcal{S}$. For all $M>0$ there exist sequences $(x_n)$ and $(q_n)$, both increasing to infinity, such that $q_n \asymp_{M,\eta} \log \log x_n$ and $$V_\eta(x_n;q_n)>Mx_n\log q_n\, .$$
   The same statement holds for $G_\eta$ in place of $V_\eta$.
\end{thm}

We next turn our attention to the $q$-average $\frac{1}{Q}\sum_{q\leq Q} V_\eta(x;q)$, which closely aligns with the weighted average considered in \cite{Density}. In \textit{loc.cit.} the authors proved, under $GRH$, that the average they consider is $\asymp x\log Q$ in the range $x^{\frac{1}{2}+\varepsilon}\leq Q\leq x$. A natural question is whether we have in general \begin{equation}\label{Avg ?}\frac{1}{Q}\sum_{q\leq Q} V_\eta(x;q)\ll_\eta x\log Q \, , \end{equation}
when $Q\leq x$ and $Q\to +\infty$ ? We note that since we are averaging $V_\eta$, this is weaker than \eqref{HConjAn}.
Fiorilli and Martin \cite{FiMa}*{Theorem 1.5} proved that the $q$-average for $V_\Lambda$ and $G$ can be as large as $x(\log \log \log x)^2$ in the range $Q\leq \sqrt{\frac{\log \log x}{\log \log \log x}} $. However we will see (in Corollary~\ref{bestbounds}) that if $GRH$ holds, then for all $q\leq x$ we have $V_\eta(x;q)\ll x (\log q)^2$. Thus for all $Q\leq x$ we have $$\frac{1}{Q}\sum_{q\leq Q} V_\eta(x;q) \ll x(\log Q)^2\,  .$$
Therefore, the $q$-average of $V_\eta$ cannot be as large as the $q$-average of $V_\Lambda$. 
Nevertheless, our second main result states that~\eqref{Avg ?} does not hold in the range $Q\asymp \sqrt{\log \log x}$. 
\begin{thm}\label{Moyenne_simplifié}
Let $\eta \in \mathcal{S}$. For all $M>0$ there exist sequences $(x_n)$ and $(Q_n)$, both increasing to infinity, such that $Q_n \asymp_{M,\eta} \sqrt{ \log \log x_n}$ and $$\frac{1}{Q_n}\sum_{Q_n<q\leq 2Q_n}V_\eta(x_n,q)>Mx_n\log Q_n\, .$$
The same statement holds for $G_\eta$ in place of $V_\eta$.
\end{thm}

Theorems~\ref{Hooley-False} and \ref{Moyenne_simplifié} will, in fact, be consequences of more precise but technical results that describe the growth behavior of $V_\eta$ and $G_\eta$ in terms of the relative sizes of $q$ and $x$. 
The following result generalizes Theorem~\ref{Hooley-False}.
\begin{thm}\label{princ} Let $\eta \in \mathcal{S}$.
    Let $h:\R_{>0}\to\R_{>0}$ be an increasing function, with $\underset{x\to +\infty}{\lim} h(x)=+\infty$, satisfying for all $t\in \R$ :
    \begin{equation}\label{property}
        h\left(e^{t^A}\right)\ll_{A}h(e^t)
    \end{equation}
    \begin{enumerate}
        \item If for all $x\geq e^3$ we have $$\frac{\log \log x}{\log \log \log x}\leq h(x)\leq \log \log x\, ,$$then, for a positive proportion of moduli $q$ there exists $x_q$ with $h(x_q)\asymp_\eta q$ such that $$V_\eta(x_q;q)\gg_{\eta}  \frac{\log \log x_q}{ q}x_q \log q\,.$$
        \item If for all $x\geq e^3$ we have $$h(x)\leq \frac{\log \log x}{\log \log \log x}\, , $$then, for a positive proportion of moduli $q$ there exists $x_q$ with $h(x_q)\asymp_\eta q$ such that $$V_\eta(x_q;q)\gg_{\eta} x_q (\log q)^2\,.$$
    \end{enumerate}
    The same statements hold for $G_\eta(x_q; q)$ in place of $V_\eta(x_q;q)$.
\end{thm}
We observe that case (1) of Theorem \ref{princ} is similar to case (1) of \cite{FiMa}*{Theorem 1.3}. However, concerning case (2) Fiorilli and Martin proved that $V_\Lambda(x_q;q) \gg x_q (\log q+\log \log \log x_q)^2$. This discrepancy is expected since our weighted variance satisfies the $GRH$ bound $$V_\eta(x;q)\ll_\eta x (\log q)^2\, ,$$
which is, thus, optimal by case (2).

Similarly for the $q$-average $\frac{1}{Q}\sum_{q\leq Q} V_\eta(x;q)$ one might ask whether the estimate  $$\frac{1}{Q}\sum_{q\leq Q} V_\eta(x;q) \ll x(\log Q)^2\,  ,$$ is best possible. The following statement, which generalizes Theorem~\ref{Moyenne_simplifié}, provides an affirmative answer to this question. We note that this Theorem is similar to \cite{FiMa}*{Theorem 1.5}, but it offers a wider range for $V_\eta$ and $G_\eta$ compared to the range given in \cite{FiMa} for the classical variances $V_\Lambda$ and $G$.

    \begin{thm}\label{moyenne} Let $\eta \in \mathcal{S}$.
    Let $h:\R_{>0}\to\R_{>0}$ be as in Theorem~\ref{princ} and assume moreover that there exists $A_0>0$ such that for all $t\in \R$
    \begin{equation}\label{minoration_h}
        h\left(e^{t^{A_0}}\right)\geq 2h(e^t)\, .
    \end{equation}
    \begin{enumerate}
        \item If for all $x\geq e^3$ we have $$\sqrt{\frac{\log \log x}{\log \log \log x}}\leq h(x) \leq \sqrt{ \log \log x}\, ,$$then, for all $Q\geq 3$ there exists $x_Q$ with $h(x_Q)\asymp_\eta Q$ such that $$\frac{1}{Q}\sum_{Q<q\leq 2Q} V_\eta(x_Q;q)\gg_{\eta} \frac{\log \log x_Q}{Q^2} x_Q \log Q \,.$$
        \item If for all $x\geq e^3$ we have $$h(x)\leq \sqrt{\frac{\log \log x}{\log \log \log x}}\, , $$then, for all $Q\geq 3$ there exists $x_Q$ with $h(x_Q)\asymp_\eta Q$ such that $$\frac{1}{Q}\sum_{Q<q\leq 2Q} V_\eta(x_Q;q)\gg_{\eta} x_Q (\log Q)^2\,.$$
    \end{enumerate}
    The same statements hold for $G_\eta(x_q; q)$ in place of $V_\eta(x_q;q)$.
\end{thm}

Let us briefly outline the structure of this paper. Section~\ref{preliminary} focuses on establishing preliminary results and introducing essential tools. In this section we will primarily study, for all $x>0$ and $\chi \in \mathcal{X}_q$, the set of Dirichlet characters modulo $q$, weighted versions of the classical prime counting functions $\psi(x;\chi)$ and $\theta(x,\chi)$ associated with $\chi$, namely: \begin{align}\label{psi}
    \psi_\eta(x,\chi)&:=\sum_{n\geq1}\Lambda(n)\chi(n)\eta\left(\frac{n}{x}\right),\\
    \theta_\eta(x,\chi)&:=\sum_{\substack{p\geq1}}\chi(p) \eta\left( \frac{p}{x}\right) \log p \, .
\end{align} 
Fiorilli and Martin \cite{FiMa}*{Theorem 1.2} proved that for any non-principal character $\chi$ modulo $q$, one has $$\frak{R}(\psi(x,\chi))=\Omega_{-}\left(x^{\frac{1}{2}}\log \log \log x \right)\, .$$ This type of estimates cannot be extended to the smooth case. Indeed, under $GRH$, we will see (Corollary~\ref{bestbounds}) that $\psi_\eta(x,\chi)\ll_{\eta,q} x^{\frac{1}{2}}$. We will relate $V_\eta$ to $\psi_\eta$ and $G_\eta$ to $\theta_\eta$ via the Parseval identities~\eqref{Parseval}, which explains why $V_\eta$ and $G_\eta$ cannot attain values as large as those of $V_\Lambda$ and $G$. Section~\ref{preliminary} also presents Hypothesis $GRH_\eta$ which is a weaker variant of the Generalized Riemann Hypothesis $(GRH)$. We then proceed to differentiate between two cases:
\begin{enumerate}
    \item The case where \(GRH_\eta\) is not satisfied is addressed in Section~\ref{NoGRH}. We use a theorem by Kaczorowski and Pintz~\cite{KP}, an improved version of Landau's Theorem, to prove that $\psi_\eta(x,\chi)$ can be as large as $x^{\theta-\varepsilon}$ for certain $\theta \in (\tfrac{1}{2},1)$ and a sufficiently small $\varepsilon>0$. Following the approach in~\cite{FiMa}, we deduce Proposition~\ref{GRH-False}.
    \item The case where \(GRH_\eta\) holds is examined in Section~\ref{GRH}. Our approach primarily employs the explicit formula~\eqref{expl}, where, under $GRH_\eta$, the sums are reduced to summing only over zeros $\rho$ such that $\frak{R}(\rho) \leq \frac{1}{2}$. One of the main difficulties is that zeros $\rho$ with $\frak{R}(\rho)<\frac{1}{2}$ may give a large contribution in general. To overcome this problem, we will prove (Lemma~\ref{Riemann-Hyp}) that $GRH_\eta$ implies that inside bounded sets there are only finitely many zeros of all Dirichlet $L$-functions that contradict the Generalized Riemann Hypothesis and we will use \cite{Topics}*{Theorem 12.1} to prove that their contribution cannot be large. This will enable us to prove that the main contribution is coming from zeros on the critical line. We then use Diophantine approximation to synchronize the summands, culminating with Proposition~\ref{GRH-True}.
\end{enumerate}

Finally in Section~\ref{proofs}, we combine Propositions~\ref{GRH-False} and~\ref{GRH-True} to deduce our main theorems.

\section{preliminary results}\label{preliminary}

For $\eta \in \mathcal{S}$, we define its Mellin transform at $s\in \C$ such that $\frak{R}(s)>-1$ by
\begin{equation}\label{DefMellin}
    \left \{ \M \right\}(s)=\mathcal{M}\eta(s):=\int_0^{+\infty}\eta(t)t^{s-1}dt\, .
\end{equation} This is a holomorphic function on $\{\frak{R}(s)>-1\}$ (by \eqref{support}) which has the following properties: 

\begin{lem}\label{Maj} Let $\eta \in \mathcal{S}$.
    \begin{enumerate}
        \item For all $s\in \C\setminus \{0\}$ with $\frak{R}(s)>-1$ we have: $$\left|\M(s)\right|\leq \frac{1}{|s(s+1)|}\int_0^{+\infty} |\eta''(t)|t^{\mathfrak{R}(s)+1}dt$$
        \item For all $s\in \C \setminus \{0\}$ such that $-\frac{1}{2}\leq \mathfrak{R}(s)\leq 2$ we have $$\left|\M(s)\right|\leq \frac{B_{\eta}}{|s|^2}\, ,$$
        where
        \begin{equation}\label{B_eta}
            B_\eta:=\underset{-\frac{1}{2}\leq \sigma \leq 2}{\max}\int_0^{+\infty}|\eta''(t)|t^{\sigma+1}dt
        \end{equation}
    \end{enumerate}
\end{lem}
\begin{proof}
    \begin{enumerate}
        \item By double integration by parts we see that for all $s\in \C\setminus\{0\}$ such that $\frak{R}(s)>-1$ we have $$ \M(s)=\frac{1}{s(s+1)}\int_0^{+\infty}\eta''(t)t^{s+1}dt\, . $$
        Hence, the first inequality is a direct consequence of the triangular inequality.
        \item For $-\frac{1}{2}\leq \mathfrak{R}(s)\leq 2 $ we have $|s|\leq |s+1|$. The second inequality is a consequence of the first one.
    \end{enumerate}
\end{proof}

We apply \cite{Mellin}*{Theorem 2.1.1} to see that:
\begin{cor}\label{InverseMellinFormula} Let $\eta \in \mathcal{S}$.
   For all $x>0$ and for all $a>-1$ we have the inversion formula: $$\eta(x)=\frac{1}{2i\pi}\int_{a-i\infty}^{a+i\infty} \M(s) x^{-s}ds$$
\end{cor}

Note that under these hypotheses on $\eta$, for any Dirichlet character $\chi$ we can define $\psi_\eta (x,\chi) $ by \eqref{psi}, and it is a continuous function on $\R_{>0}$. 
\begin{lem}\label{redprimLemma}
   Let $\eta \in \mathcal{S}$ and let $\chi$ be a Dirichlet character modulo $q$ with conductor $q_\chi$, and let $\chi^*$ be the primitive character modulo $q_\chi$ inducing $\chi$. Then \begin{equation}\label{redprim}
        \psi_\eta\left(x,\chi^*\right)-\psi_\eta(x,\chi) \ll_\eta \log q \, .
    \end{equation} 
\end{lem}
\begin{proof}
    
    For a prime $p$, by \eqref{support}, we have: 
    \begin{align*}
        \sum_{k\geq 1} \eta\left(\frac{p^k}{x}\right) &\ll_\eta \sum_{\substack{k\geq 1 \\ p ^k \leq x}} \frac{p^k}{x}+\sum_{\substack{k\geq 1\\ p^k >x}} \frac{x}{p^k} \\
        &\ll \sum_{k\geq 0}\frac{1}{p^k} \ll 1 \, .
    \end{align*}
    Thus \begin{align*}
        \psi_\eta(x,\chi^*)-\psi_\eta(x,\chi) &= \sum_{\substack{p|q \\ (p,q_\chi)=1}} \log p \sum_{k\geq 1} \eta\left(\frac{p^k}{x}\right)\chi(p^k) \\
        &\ll_\eta  \sum_{\substack{p|q }} \log p \leq \log q \, .
    \end{align*}
\end{proof}
The following Proposition is a smooth version of Perron's formula: 
\begin{prop}\label{InverseMellin}
Let $\eta \in \mathcal{S}$ and let $\chi$ be a Dirichlet character modulo $q$.
\begin{enumerate}
    \item For $s\in \C$ with $\mathfrak{R}(s)>1$ we have: \begin{equation}\label{15}
        -\frac{L'}{L}(s,\chi) \M(s)=\int_0^{+\infty}\psi_\eta(x,\chi)x^{-(s+1)}dx \, .
    \end{equation}
    \item For all $x\geq 1$ we have: \begin{equation}\label{16}
        \psi_\eta(x,\chi)=-\frac{1}{2i\pi}\int_{2-i\infty}^{2+i\infty}\frac{L'}{L}(s,\chi)\M(s)x^sds \, .
    \end{equation}
\end{enumerate}
    
\end{prop}
\begin{proof}
Fix $s\in \C$ with $\sigma:=\mathfrak{R}(s)>1$. We start by noticing that for all $n\geq1$, a change of variables implies that: \begin{equation*}
            \M(s)=\int_0^{+\infty}\eta\left(\frac{n}{x}\right)\frac{n^s}{x^{s+1}}dx \, .
        \end{equation*}
Thus, \begin{align*}
          \sum_{n\geq1}\frac{\chi(n)\Lambda(n)}{n^s}\M(s)=\sum_{n\geq1}\int_0^{+\infty}\chi(n)\Lambda(n)\eta\left(\frac{n}{x}\right)\frac{1}{x^{s+1}}dx \, .
       \end{align*}
Hence, to prove \eqref{15}, it is sufficient to interchange the sum and integral.\\
To do so, it suffices to prove that: \begin{equation*}
            \sum_{n\geq1}\int_0^{+\infty}\left|\chi(n)\Lambda(n)\eta\left(\frac{n}{x}\right)\frac{1}{x^{s+1}}\right|dx\leq  \sum_{n\geq1}\log n \int_0^{+\infty}\eta\left(\frac{n}{x}\right)\frac{1}{x^{\sigma+1}}dx < +\infty \, .
        \end{equation*}
Let $n \geq 1$ and consider the integral \begin{equation}\label{twointeg}
            \int_0^{+\infty}\eta\left(\frac{n}{x}\right)\frac{1}{x^{\sigma+1}}dx=  \int_0^{n}\eta\left(\frac{n}{x}\right)\frac{1}{x^{\sigma+1}}dx+\int_n^{+\infty}\eta\left(\frac{n}{x}\right)\frac{1}{x^{\sigma+1}}dx \, .
        \end{equation}
We use the estimation $\eta(\frac{n}{x}) \ll \left(\frac{x}{n}\right)^{\sigma+1}$ to see that the first integral in \eqref{twointeg} is $$\ll_\eta \frac{1}{n^{\sigma}}\, .$$
For the second integral we only need $\eta$ to be bounded which is trivially true thanks to \eqref{support}. Hence, the second integral is $$\ll_\eta \frac{1}{n^\sigma}\, . $$
This proves \eqref{15}.\\
We now move to (2). Let $x\geq1$. By an inverse Mellin transform we have: 
\begin{align*}
            \psi_\eta(x,q)&=\sum_{n\geq1}\Lambda(n)\chi(n)\eta\left(\frac{n}{x}\right)\\
            &=\sum_{n\geq1} \chi (n) \Lambda(n) \frac{1}{2i\pi}\int_{2-i\infty}^{2+i\infty}\M(s)\left(\frac{n}{x}\right)^{-s}ds\\
            &=-\frac{1}{2i\pi}\int_{2-i\infty}^{2+i\infty}\frac{L'}{L}(s,\chi)\M(s)x^sds\, ,
\end{align*}
Where we interchanged summation and integration thanks to Lemma \ref{Maj} which implies that $$\sum_{n\geq1} \Lambda(n) \int_{2-i\infty}^{2+i\infty}\left|\M(s)\left(\frac{n}{x}\right)^{-s}\right | |ds|<+\infty \, .$$
This proves~\eqref{16}
   
\end{proof}
We are now ready to state and prove the explicit formula for our weighted prime counting functions:
\begin{prop}[Explicit formula]\label{explform}
Let $\eta \in \mathcal{S}$ and let $\chi$ be a Dirichlet character modulo $q\geq 1$. Uniformly for $x\geq 1$ we have: 
\begin{equation}\label{expl}
    \psi_\eta(x,\chi)=\mathbf{1}_{\chi=\chi_0}\cdot\M(1)x-\sum_{\rho_\chi} \M(\rho_\chi)x^{\rho_\chi}+O_\eta(\log(q+1))\, ,
\end{equation}
    where the sum runs over all non-trivial zeros of $L(s,\chi)$ and where $\chi_0$ is the principal character modulo $q$.
\end{prop}
\begin{proof}
    First assume that $\chi$ is primitive.
    By Proposition \ref{InverseMellin}(2) we have for all $x\geq 1$ : \begin{equation}
        \psi_\eta(x,\chi)=-\frac{1}{2i\pi}\int_{2-i\infty}^{2+i\infty}\frac{L'}{L}(s,\chi)\M(s)x^sds \, .
    \end{equation}
    Denote $\kappa$ to be $0$ if $\chi$ is even and $1$ if $\chi$ is odd. Let $T\geq2$ and $N$ a large positive integer which is even if $\chi$ is odd and odd if $\chi$ is even. 
    Denote $$f_{\chi}(s)=\frac{L'}{L}(s,\chi)\M(s)x^s\, .$$
    By \cite{MV}*{Lemma 12.7} there exists $T\leq T_1\leq T+1$ such that uniformly for $-1\leq \sigma \leq 2$ $$\frac{L'}{L}(\sigma+iT_1)\ll \log(qT)^2 \, .$$

We consider the rectangular contour $\mathcal{C}$ from \(2 - iT_1\) to \(-\frac{1}{2}- iT_1\), then from \(-\frac{1}{2} - iT_1\) to \(-\frac{1}{2} + iT_1\), next from \(-\frac{1}{2} + iT_1\) to \(2 + iT_1\), and finally from \(2 + iT_1\) back to \(2 - iT_1\) (which is negatively oriented). Given that the poles of \(f_\chi(s)\) inside this contour are the non-trivial zeros \(\rho_\chi\) of \(L(s, \chi)\) satisfying \(\mathfrak{I}(\rho_\chi) < T_1\) and an additional pole at $s=0$ when $\chi$ is odd, the application of the residue Theorem yields:

\[
\frac{1}{2i\pi}\oint_{\mathcal{C}} f_\chi(s)\, ds=-\left(-\mathbf{1}_{\chi=\chi_0}\cdot\M(1)x+\sum_{\substack{\rho_\chi \\ \mathfrak{I}(\rho_\chi) < T_1}} \M(\rho_\chi) x^{\rho_\chi} + (1-\kappa) \M(0) \right)\, .
\]
By \cite{MV}*{Lemma 12.7} and Lemma \ref{Maj}(2) we have $$\left |\int_{2-iT_1}^{-\frac{1}{2}-iT_1}f_\chi(s)ds\right| \leq \int_{-\frac{1}{2}}^2\left|  \frac{L'}{L}(t+iT_1,\chi)\M(s)\right|x^tdt\ll_\eta \left(\frac{\log qT}{T}x\right)^2 \, .$$
Similarly, $$\int_{2+iT_1}^{-\frac{1}{2}+iT_1}f_\chi(s)ds\ll_\eta \left(\frac{\log qT}{T}x\right)^2\, .$$
By \cite{MV}*{Lemma 12.6}, we see that for all $-T_1\leq t \leq T_1$ we have $\frac{L'}{L}(-\frac{1}{2}+it)\ll \log(q(|t|+2)$. Thus,  by Lemma \ref{Maj}(2) $$\int_{-\frac{1}{2}-iT_1}^{-\frac{1}{2}+iT_1}\frac{L'}{L}(s,\chi)\M(s)x^s ds\ll \frac{B_\eta}{\sqrt{x}}\int_{-T_1}^{T_1}\frac{\log(q(|t|+2))}{\frac{1}{4}+t^2}dt \ll B_\eta \log (q+1)\, ,$$
where $B_\eta$ was defined in Lemma~\ref{Maj}(2).\\
Thus, \begin{equation*}
    -\frac{1}{2i\pi}\int_{2-iT_1}^{2+iT_1}f_{\chi}(s)ds=\mathbf{1}_{\chi=\chi_0}\cdot\M(1)x-\sum_{\substack{\rho_\chi \\ \mathfrak{I}(\rho_\chi) < T_1}} \M(\rho_\chi) x^{\rho_\chi} +O_\eta\left( \log (q+1) +\left(x\frac{\log qT}{T}\right)^2 \right)\, .
\end{equation*}
Now letting $T \to \infty$ (which implies $T_1\to \infty$) gives \begin{equation*}
    \psi_\eta(x,\chi)=\mathbf{1}_{\chi=\chi_0}\cdot\M(1)x-\sum_{\rho_\chi} \M(\rho_\chi)x^{\rho_\chi}+O_\eta( \log (q+1))\, ,
\end{equation*}
when $\chi$ is primitive. For the general case, it suffices to apply Lemma~\ref{redprimLemma} to obtain the result.
\end{proof}
Applying the explicit formula requires an understanding of the distribution of zeros on the critical strip. De La Bretèche and Fiorilli \cite{BF2}*{Hypothesis $GRH_{\widehat{\eta}}$} stated a hypothesis slightly weaker than the Generalized Riemann hypothesis that emerges naturally when employing the explicit formula. We state a hypothesis that is very similar to theirs:
\begin{Hyp}\label{Hyp GRH}
    One says that $GRH_\eta$ holds if for any non-trivial zero $\rho_\chi$ of $L(s,\chi)$, with $\chi$ a non-principal Dirichlet character modulo $q\geq3$, such that $\mathfrak{R}(\rho_\chi) > \frac{1}{2}$ we have $\M(\rho_\chi)=0$.
\end{Hyp}

We first prove that $GRH_\eta$ implies that in compact sets there can only be finitely many exceptions to the Generalized Riemann Hypothesis.
\begin{lem}\label{Riemann-Hyp}
Let $\eta \in \mathcal{S}$, let $D>0$ and let \begin{equation*}
    \Gamma:=\left \{s\in \C\ :\ 0\leq \mathfrak{R}(s) \leq 1 \text{ with } \ \mathfrak{R}(s) \ne \frac{1}{2} \text{ and }-D \leq \mathfrak{I}(s) \leq D \right \} \, .
\end{equation*} Denote $Z$ to be the set of non-trivial zeros of all Dirichlet L-functions. If $GRH_\eta$ holds then $\Gamma \cap Z$ is finite. Moreover, there exists $0<\beta_\eta<\frac{1}{2}$ such that uniformly for $q\geq 3$ we have \begin{equation*}
        \sum_{\chi \in \mathcal{X}_q}\sum_{\rho_\chi \in \Gamma}|\M(\rho_\chi)|\ll_{ \Gamma,\eta}q^{\frac{3\beta_\eta}{1+\beta_\eta}}(\log q)^9 \, .
    \end{equation*} 
\end{lem}

\begin{proof}
    Since $\M$ is a non-zero holomorphic function on $\{\frak{R}(s)>-1\}$, it can only have finitely many zeros in $\Gamma$. If $GRH_\eta$ holds, then, by the functional equation, for all zeros $\rho $ of Dirichlet L-functions with $\frak{R}(\rho)\ne \frac{1}{2}$, either $\rho$ or $1-\overline{\rho}$ is a zero of $\M$, and there can only be finitely many of them that are inside $\Gamma$. Considering only zeros $\rho$ such that $\frak{R}(s) <\frac{1}{2}$ and taking $\beta_\eta<\frac{1}{2}$ larger than the maximum of their real parts so that $\sigma_\eta=1-\beta_\eta \leq \frac{4}{5}$, we can apply \cite{Topics}*{Theorem 12.1} along with the fact that $\M$ is bounded on $\Gamma$ to obtain the desired estimate.
\end{proof}
In order to relate $V_\eta$ to $\psi_\eta$ and $G_\eta$ to $\theta_\eta$ we use the Parseval identities
\begin{align}\label{Parseval}
    V_\eta(x;q)&=\frac{1}{\phi(q)}\sum_{\chi\ne \chi_0}\left|\psi_\eta(x,\chi)\right|^2 ,\\
    G_\eta(x;q)&=\frac{1}{\phi(q)}\sum_{\chi}|\theta_\eta(x,\chi)-\mathbf{1}_{\chi=\chi_0}\M(1)x |^2 \, ,
\end{align}
which can be proven with direct computation.
The following Corollary is a consequence of Proposition~\ref{explform}. It shows that our weighted prime counting functions cannot have large oscillations and thus the variance $V_\eta$ cannot be as large as $V_\Lambda$ when $x$ is too large compared to $q$. (See \cite{FiMa}*{Theorem 1.3, equation (11)})

\begin{cor}\label{bestbounds}
    Let $\eta \in \mathcal{S}$. Assume that $GRH_\eta$ holds. \begin{enumerate}
        \item For all $q\geq3$ and for any non-principal character $\chi$ modulo $q$, we have: \begin{equation}
        \psi_\eta(x,\chi)\ll_\eta x^{\frac{1}{2}}\log q \, .
        \end{equation}
        \item Uniformly for $x\geq q\geq3$ we have: \begin{equation}
            V_\eta(x;q) \ll_\eta x (\log q)^2 \, .
        \end{equation}
    \end{enumerate} 
    
\end{cor}
\begin{proof}
    The first estimate is a direct consequence of the explicit formula \eqref{expl}. Indeed, if $GRH_\eta$ holds then the sum in \eqref{expl} is reduced to the sum over zeros $\rho$ such that $\frak{R}(\rho)\leq \frac{1}{2}$. Moreover, thanks to Lemma \ref{Maj} we see that $$\sum_{\rho_\chi}\M(\rho_\chi)\ll_\eta \sum_{\rho_\chi}\frac{1}{|\rho_\chi|^2}\ll \log q\, .$$
    The second estimate is obtained by combining the first assertion of the current Corollary with the Parseval identity~\eqref{Parseval}.
\end{proof}

Another important consequence of the explicit formula is the following. It will enable us to relate $\psi_\eta$ with $\theta_\eta$, thus, thanks to Parseval identities, it will allow us to restrict our study to $\psi_\eta$ and $V_\eta$ and then conclude similar results for $\theta_\eta$ and $G_\eta$.
\begin{lem}\label{psi-theta}
Let $q\geq 3$ and denote $\chi_0$ the principal character modulo $q$.
\begin{enumerate}
    \item Uniformly for $x\geq 1$ we have $$\psi_\eta(x,\chi_0)\ll\left(\M(1)+\sum_{\rho}|\M(\rho)|\right)x \, ,$$ with an absolute implicit constant and where the sum runs over non-trivial zeros of the Riemann zeta function.
    \item For all Dirichlet characters $\chi$ modulo $q$, we have uniformly for $x>\max(q,8)$  $$\psi_\eta(x,\chi)-\theta_\eta(x,\chi) \ll_\eta x^{\frac{1}{2}}\, .$$
\end{enumerate}
    
\end{lem}
\begin{proof}
        The first estimate is a direct consequence of the explicit formula \eqref{expl}.\\ 
        Let us prove the second estimate. Denote, for $\ell\geq 1$, $f_\ell:\R\to \R$ defined by $f(x)=x^\ell$. We first notice that $$\psi_\eta(x,\chi)-\theta_\eta(x,\chi)=\sum_{\ell\geq 2}\theta_{\eta\circ f_\ell}\left(x^{\frac{1}{\ell}},\chi \circ f_\ell\right)\, .$$ 
        Using the first estimate of the Lemma, we see that $$\theta_{\eta\circ f_2}\left(x^{\frac{1}{2}},\chi \circ f_2 \right) \ll \theta_{\eta\circ f_2}\left(x^{\frac{1}{2}},\chi_0\right)\leq \psi_{\eta \circ f_2} \left(x^{\frac{1}{2}},\chi_0\right) \ll_\eta x^{\frac{1}{2}} \, .$$
        By a change of variables, we have for all $s\in \C$ with $\frak{R}(s)\geq0$ $$\Ml(s)=\frac{1}{\ell}\int_0^{+\infty}\eta(u)u^{\frac{s}{\ell}-1}du=\frac{1}{\ell}\M\left(\frac{s}{\ell}\right) \, .$$
        Thus $$\sum_{3\leq \ell \leq \frac{\log x}{\log 2}}\Ml (1)=\sum_{3\leq \ell \leq \frac{\log x}{\log 2}}\frac{1}{\ell}\M\left(\frac{1}{\ell}\right)\ll_\eta \sum_{3\leq \ell \leq \frac{\log x}{\log 2}}1\ll \log x \, ,$$
        and $$\sum_{3\leq \ell \leq \frac{\log x}{\log 2}} \sum_{\rho_\chi} \frac{1}{\ell}\M\left(\frac{\rho_\chi}{\ell}\right)\ll_\eta \sum_{3\leq \ell \leq \frac{\log x}{\log 2}} \sum_{\rho_\chi} \frac{\ell}{|\rho_\chi|^2}\ll (\log x)^2 \log q \, .$$
        As $q\leq x$ we have by the first estimate of the current Lemma \begin{align*}\sum_{3\leq \ell \leq \frac{\log x}{\log 2}} \theta_{\eta \circ f_\ell}\left (x^{\frac{1}{\ell}}, \chi \circ f_\ell \right) &\ll_\eta \left( \sum_{3\leq \ell \leq \frac{\log x}{\log 2}} \left( \Ml(1)+\sum_{\rho_\chi} |\Ml(\rho_\chi) |\right) \right)x^{\frac{1}{3}}\\  &\ll_\eta (\log x)^3 x^{\frac{1}{3}}\ll x^{\frac{1}{2}} \, . \end{align*}
        Moreover \begin{align*}
            \sum_{\ell>\frac{\log x}{\log 2}}\theta_{\eta\circ f_\ell}\left(x^{\frac{1}{\ell}},\chi\circ f_\ell \right)&=\sum_{\ell>\frac{\log x}{\log 2}}\sum_{p\geq2} \log p\cdot \eta\left(\frac{p^\ell}{x}\right)\chi(p^\ell)\\
            &\ll \sum_{\ell>\frac{\log x}{\log 2}}\sum_{p\geq2} \log p \cdot \frac{x}{p^\ell} \\
            &=\sum_{p\geq2}\log p\cdot \frac{x}{p^{\ell_0}}\cdot\frac{p}{p-1}\, ,
        \end{align*}
        where $\ell_0:=\left[\frac{\log x}{\log 2}\right]+1$.
        Thus $$\sum_{\ell>\frac{\log x}{\log 2}}\theta_{\eta\circ f_\ell}\left(x^{\frac{1}{\ell}},\chi\circ f_\ell \right)\ll_\eta \sum_{p\geq 2} \log p \left( \frac{2}{p}\right)^{\ell_0}\ll 1 \, .$$
        This proves the second estimate.
    
\end{proof}

\section{ Case where $GRH_\eta$ does not hold}\label{NoGRH}

The goal of this section is to prove that if $GRH_\eta$ does not hold for some non-principal character modulo $q\geq 3$ then the estimates~\eqref{HConjAn} do not hold in any range $q\asymp x^{o(1)}$.
Note that thanks to Lemma \ref{InverseMellin} and Proposition \ref{explform}, all the results of \cite{FiMa}*{\S 2} hold. The following Lemma shows that when $GRH_\eta$ does not hold then $\psi_\eta(x,\chi)$ and $\theta_\eta(x,\chi)$ can have large values. We note that it can be proven as in \cite{FiMa}*{Lemma 2.1}, however in this paper we give a slightly different proof.

\begin{lem}Let $\eta \in \mathcal{S}$ and let $\chi$ be a non-principal character modulo $q\geq3$. Assume that $$\theta_\chi :=\sup\{\mathfrak{R}(\rho_\chi)\ :\ L(\rho_\chi,\chi)=0\ \text{and}\ \M(\rho_\chi)\ne 0\}>\frac{1}{2}\, .$$ 
For all $\varepsilon>0$, there exists $X_0>1$ such that for all $X>X_0$ there exists $x\in [X^{1-\varepsilon},X]$ for which : $$\frak{R}\left(\psi_\eta(x,\chi)\right)<- x^{\theta_\chi- \varepsilon} $$
   
\end{lem}
\begin{proof}
Denote : $$F(s):=\frac{-\M(s)}{2}\left(\frac{L'}{L}(s,\chi)+\frac{L'}{L}(s,\overline{\chi})\right)\, .$$
First notice that $\M(\sigma)>0$ for $\sigma \in \R$ and thus any pole $\sigma \in \R$ of $F$ is simple with a negative residue that we denote $m_\sigma:=\text{Res}_{s=\sigma}F(s)$. Denoting $$\mathcal{P}:=\{\sigma \in [\tfrac{1}{2},1[\ :\ \sigma\text{ is a pole of }F\}\, ,$$
which is finite, we can define: $$G(s):=F(s)-\sum_{\sigma \in \mathcal{P}}\frac{m_\sigma}{s-\sigma}\, .$$
By Lemma \ref{InverseMellin} we have for all $\mathfrak{R}(s)>1$ : $$ F(s)=\int_0^{+\infty}\mathfrak{R}(\psi_\eta(x,\chi))x^{-(s+1)}dx\, .$$
Thus, for all $\mathfrak{R}(s)>1$ $$G(s)=\int_0^{+\infty}\left(\mathfrak{R}(\psi_\eta(x,\chi))-\mathbf{1}_{[1,+\infty)}(x)\left(\sum_{\sigma \in \mathcal{P}}m_\sigma x^\sigma \right)\right)x^{-(s+1)}dx$$
Denote $$\theta_\chi':=\sup\{\mathfrak{R}(\rho_\chi)\ :\ \rho_\chi \text{ is a pole of } G\}\geq\frac{1}{2}\, .$$
Let $\varepsilon>0$, by \cite{KP}*{Theorem 1} there exists  $X_0>1$ such that for all $X>X_0$ there exists $x\in [X^{1-\varepsilon},X]$ for which $$\mathfrak{R}(\psi_\eta(x,\chi))-\left(\sum_{\sigma \in \mathcal{P}}m_\sigma x^\sigma \right)<-x^{\theta_\chi'-\varepsilon}\, .$$
Since for all $\sigma \in \mathcal{P}$ we have $m_\sigma < 0$, then noting that $\theta_\chi=\max(\mathcal{P}\cup \{\theta'_\chi\})$, we have (up to considering a larger $X_0$):
$$\mathfrak{R}(\psi_\eta(x,\chi))<-x^{\theta_\chi-\varepsilon}\, .$$

\end{proof}

Thanks to the previous Lemma, the following proposition is similar to \cite{FiMa}*{Proposition 2.3} and can be proven in the exact same way. 

\begin{prop}\label{GRH-False} Let $\eta \in \mathcal{S}$.
    If $GRH_\eta$ does not hold, then there exists $\delta>0$ and $q_0\geq 3 $ (both depending on $\eta$) such that for all increasing functions $h:\R_{>0}\to\R_{>0}$ satisfying $\underset{x\to +\infty}{\lim} h(x)=+\infty$ and $h(x)=o(x^{\delta})$, as $x\to \infty$, and for all large multiple $q$ of $q_0$, there exists $x_q \in [h^{-1}(q), h^{-1}(q) ^{\frac{1}{1-\delta}}]$ such that $$V_\eta(x_q;q) \geq x_q ^{1+\delta}\,.$$
    The same statement holds for $G_\eta(x_q; q)$ in place of $V_\eta(x_q;q)$.
\end{prop}

\section{ The result under $GRH_\eta$}\label{GRH}

The aim of this section is to prove Proposition~\ref{GRH-True} below. We state it in a way that includes both Theorems~\ref{princ} and~\ref{moyenne}. 
We first start by using Weil explicit formula to obtain an estimation of the sums $\sum_{\rho_\chi}\M(\rho_\chi)$ over zeros of Dirichlet $L$-function. Combining this estimation with Lemma~\ref{Riemann-Hyp} will enable us to reduce the study of these sums to sums over zeros inside a fixed bounded set that will be well chosen, which will be crucial to prove our main results.

\begin{prop}\label{Weil}
Let $\eta \in \mathcal{S}$ and let $\chi$ be a non-principal Dirichlet character modulo $q\geq3$. Then, $$\sum_{\rho_\chi}\M(\rho_\chi)=\underset{T\to +\infty}{\lim}\underset{|\mathfrak{I}(\rho_\chi)|\leq T}{\sum}\M(\rho_\chi)=\eta(1)\log q_\chi +O_\eta(1) \, .$$
    
\end{prop}
\begin{proof}
    
    This is a consequence of Weil's explicit formula (see \cite{MV}*{Theorem 12.13}). Denote $$F_\eta(x):=e^{\pi x}\eta\left(e^{2\pi x}\right) \, .$$ Using \eqref{support} we see that the function $F_\eta$ satisfies the hypotheses of \cite{MV}*{Theorem 12.13}. Moreover, denoting as in \cite{MV}*{Theorem 12.13} $$ \Phi(s):=\int_{-\infty}^{+\infty} F_\eta(x) e^{-2\pi x(s-\frac{1}{2})}dx=\frac{1}{2\pi}\M(1-s)\, ,$$
    and using the functional equation to see that when $\rho_\chi$ runs over zeros of $L(s,\chi)$ we have $1-\rho_\chi$ runs over zeros of $L(s,\overline{\chi})$. Finally, applying \cite{MV}*{Theorem 12.13}, we obtain the result (where we used \eqref{support} to bound the terms other than $\eta(1)\log q_\chi$ and used the fact $q_\chi=q_{\overline{\chi}}$).
\end{proof}
The following lemma is a direct consequence of \cite{FiMa}*{Lemma 3.2} where the authors state a more precise result:
\begin{lem}\label{subset}
    There exists $B>0$ such that for all $q>B$ there exists a subset $\mathcal{F}_q \subset \mathcal{X}_q$ of the set of all characters modulo $q$ with cardinality $\left |\mathcal{F}_q \right| \geq \frac{\phi(q)}{2}$, such that for all $\chi \in \mathcal{F}_q$ we have
    \begin{equation}
       \log q_\chi > \log q- (\log \log q)^2 > \frac{3}{4} \log q  \, .
    \end{equation}
\end{lem}
Let us recall the Riemann - von Mangoldt formula, which states that for all $T\geq2$:  \begin{equation}\label{RVM}
    N(T,\chi):=|\{\rho_\chi\ :\ |\mathfrak{I}(\rho_\chi)|\leq T\}|=\frac{T}{\pi}\log\left(\frac{q_\chi T}{2\pi e}\right)+O(\log(qT))\, .
\end{equation}
For $\mathcal{G}_q\subset \mathcal{X}_q$ a set of characters $\chi$ with conductor $q_\chi$ satisfying $\log q_\chi>\log q - (\log \log q)^2$, we denote $$N(T,\mathcal{G}_q):=\sum_{\chi \in \mathcal{G}_q}N(T,\chi) \, .$$
Denoting $\Phi_q=|\mathcal{G}_q|$ and $E(\mathcal{G}_q):=\sum_{\chi \in \mathcal{G}_q}\log q-\log q_\chi$, we see that $E(\mathcal{G}_q)<\Phi_q (\log \log q)^2$. Thus; with same computation as in \cite{FiMa}*{page 4804}, we obtain
\begin{equation}\label{Riem-VM-sets}
    N(T,\mathcal{G}_q)=\frac{\Phi_q}{\pi}T\log(qT)+O\left(E(\mathcal{G}_q)T+\Phi_q(T+\log  qT)\right) \,  .
\end{equation}
An easy consequence of the Riemann - von Mangoldt formula is the following.
\begin{lem}\label{Riem-VM}
    There exists an absolute constant $C>0$ such that for all $T\geq 2$, for all $q\geq 3$ and for all Dirichlet characters $\chi$ modulo $q$ \begin{align*}
     \sum_{|\mathfrak{I}(\rho_\chi)| > T} \frac{1}{|\rho_\chi|^2} &\leq C \frac{\log(qT)}{T}\, , \\
     \sum_{\rho_\chi}\frac{1}{|\rho_\chi|^2}&\leq C  \log q\, .
 \end{align*}
\end{lem}

Let $C>0$ be as in Lemma~\ref{Riem-VM} and denote $H_\eta=C\cdot B_\eta$ (where $B_\eta$ was defined in Lemma~\ref{Maj}(2)). Applying Lemmas~\ref{Maj} and \ref{Riem-VM} we see that for all $q$ large enough and all $\chi  \in \mathcal{G}_q$ \begin{equation}\label{Cor:Riem-VM}\underset{\rho_\chi}{{\sum}} \left |\M(\rho_\chi) \right|\leq B_\eta \sum_{\rho_\chi} \frac{1}{|\rho_\chi|^2} \leq H_\eta \log q\, .\end{equation}
We denote \begin{equation}\label{C_eta}
     C_\eta :=\log \left ( \left[16\pi^2 \frac{H_\eta}{\eta(1)}\right]+1 \right)\, ,
 \end{equation}
 which will be useful in the statement of Proposition~\ref{GRH-True}.

The Diophantine approximation result we will use to prove our main results is an easy consequence of \cite{FiMa}*{Lemma 3.7}; we state it here for completeness.
\begin{lem}\label{Dio}
    Let $\Lambda$ be a set of real numbers of cardinality $k \geq 2$ and let $M\geq 2$ be an integer. Then, for all $N \geq M^{3k}$, there exists an integer $n$ such that $N^{\frac{1}{3}}<n\leq N$ and for all $\lambda \in \Lambda$, we have $$\| n\lambda \| \leq \frac{1}{M}\, ,$$ where $\|x\|$ denotes the distance of $x\in \R$ from the nearest integer.
    \end{lem}
\begin{proof} Using \cite{FiMa}*{Lemma 3.7} we have $$\# \left\{ n\leq N\ :\ \forall\lambda\in \Lambda\ \|n\lambda\| \leq \frac{1}{M} \right\} \geq \frac{N}{M^k}-1 \, .$$
Thus, for $N \geq M^{3k}$, we have $$\# \left\{ n\leq N\ :\ \forall\lambda\in \Lambda\ \|n\lambda\| \leq \frac{1}{M} \right\} >N^{\frac{1}{3}} \, ,$$which implies the lemma.   
\end{proof}

The following lemma is the main technical tool of this paper.

\begin{lem}\label{MainLemma}
    Let $\eta \in \mathcal{S}$. For all $Q\geq3$ let $\mathcal{D}_Q $ be a non-empty subset of $ [Q,2Q]\cap \N$, let $\mathcal{G}_Q$ be a subset of characters $\chi \in \mathcal{X}_Q$ with conductor $q_\chi$ satisfying $\log q_\chi>\log Q- (\log \log Q)^2$, and assume that its cardinal $\Phi_Q \gg \frac{\phi(Q)}{\log Q}$, and let $f:\N \to \R_{>C_\eta}$, where $C_\eta$ is defined in \eqref{C_eta}. If $GRH_\eta$ holds, then for all $Q$ large enough in terms of $\eta$ there exist $x_Q$ such that \begin{align*}
        \log \log x_Q \asymp_\eta& \left( \sum_{q\in \mathcal{D}_Q}\Phi_q \right)f(Q)\log Q \, ,\\
        \left|\sum_{q\in \mathcal{D}_Q}\sum_{\chi \in \mathcal{G}_q}\psi_\eta(x_Q,\chi) \right|& \gg_\eta \left(\sum_{q\in \mathcal{D}_Q}\Phi_q\right) \sqrt{x_Q}\log Q \, .
    \end{align*}
\end{lem}
\begin{proof}
By \eqref{Riem-VM-sets} we have uniformly for $Q\geq3$ $$\sum_{q\in \mathcal{D}_Q}N(T,\mathcal{G}_q)=\sum_{q\in \mathcal{D}_Q}\frac{\Phi_q}{\pi}T\log(qT)+O\left(\sum_{q\in \mathcal{D}_Q}(E(\mathcal{G}_q)T+\Phi_q(T+\log  qT))\right) \,  .$$
By assumption on $q_\chi$ for $\chi \in \mathcal{G}_q$, there exists an absolute constant $R>0$ large enough so that uniformly for $T,Q \geq R$ we have
\begin{equation}\label{ Riemm-VM2} \sum_{q\in \mathcal{D}_Q}N(T,\mathcal{G}_q)\asymp \sum_{q\in \mathcal{D}_Q}\Phi_qT\log(QT)\, .\end{equation}
Lemmas~\ref{Maj}(2) and \ref{Riem-VM} allow us to consider $D_\eta>R$ large enough so that for all $q$ large enough and all $\chi  \in \mathcal{G}_q$ we have
\begin{equation}\label{Restmaj}
    \sum_{|\mathfrak{I}(\rho_\chi)|> D_\eta}\left|\M(\rho_\chi)\right|< \frac{\eta(1)}{16} \log q \, .
\end{equation}
Let us consider the set \begin{equation}\label{Rectangle}
    \Gamma:=\left \{s\in \C\ :\ 0\leq \mathfrak{R}(s) \leq 1 \text{ with } \ \mathfrak{R}(s) \ne \frac{1}{2} \text{ and }-D_\eta \leq \mathfrak{I}(s) \leq D_\eta \right \}
\end{equation} 
By Proposition~\ref{Weil} we have for all $q$ large enough and all $\chi  \in \mathcal{G}_q$ $$\sum_{\rho_\chi}\M\left(\rho_\chi\right)=\eta(1)\log q_\chi +O_\eta(1)\,  .$$
Using Lemma~\ref{Riemann-Hyp} we see that there exists $\alpha=\alpha(\eta)<1$ such that uniformly for $q\geq3$ we have $$\sum_{ \chi \in \mathcal{G}_q} \sum_{\rho_\chi \in \Gamma} | \M (\rho_\chi) | \ll_\eta q^\alpha .$$
Hence $$\sum_{\chi \in \mathcal{G}_q}\underset{\substack{|\mathfrak{I}(\rho_\chi)|\leq D_\eta,\\ \rho_\chi \notin \Gamma}}{ {\sum}}\M(\rho_\chi)=\eta(1)\sum_{\chi \in \mathcal{G}_q}\log q_\chi-\sum_{|\mathfrak{I}(\rho_\chi)|> D_\eta}\M(\rho_\chi)+O_\eta(\Phi_q)+O_\eta(q^\alpha)\, .$$
Combining the fact $\Phi_q \gg \frac{\phi(q)}{\log q} \gg q^\alpha$ with \eqref{Restmaj}, and using our assumption assumption on $q_\chi$ for $\chi \in \mathcal{G}_q$, we see that for all $Q$ large enough \begin{equation}\label{R_eta}
     \left|\sum_{q\in \mathcal{D}_Q}\sum_{\chi \in \mathcal{G}_q}\underset{\substack{|\mathfrak{I}(\rho_\chi)|\leq D_\eta,\\ \rho_\chi \notin \Gamma}}{{\sum}}\M(\rho_\chi)\right| >\frac{3}{8}\eta(1) \sum_{q\in \mathcal{D}_Q} \Phi_q \log q\, .    
 \end{equation}
For $Q$ large enough in terms of $\eta$, using Lemma~\ref{Dio}, with $M=\exp(C_\eta)$ and $$N:=\exp\left(3f(Q)\sum_{q\in \mathcal{D}_Q}N(D_\eta,\mathcal{G}_q)\right)\, ,$$ we see that there exists $t_Q\in [N^\frac{1}{3},N]$ such that for all $q\in \mathcal{D_Q}$ and $\chi \in \mathcal{G}_q$ and all $\rho_\chi$ with $|\mathfrak{I}(\rho_\chi)|\leq D_\eta$ we have $ \left\| \frac{\mathfrak{I}(\rho_\chi) t_q}{2\pi} \right\| \leq e^{-C_\eta}$.
Thus, by \eqref{Cor:Riem-VM}
\begin{equation*}
    \left |\sum_{q\in \mathcal{D}_Q}\sum_{\chi \in \mathcal{G}_q} \underset{|\mathfrak{I}(\rho_\chi)|\leq D_\eta}{{\sum}}  \M(\rho_\chi)(1-e^{i\mathfrak{I}(\rho_\chi) t_q})\right|\leq \frac{4\pi^2}{e^{C_\eta}} H_\eta \sum_{q\in \mathcal{D}_Q}\Phi_q  \log q   < \frac{\eta(1)}{4} \sum_{q\in \mathcal{D}_Q} \Phi_q \log q  \, ,
\end{equation*}
where we used the fact $$|1-e^{it}| \leq |1-\cos(t)|+|\sin(t)| \leq \frac{1}{2}\left(2\pi  \left \|\frac{t}{2\pi} \right\| \right)^2+2\pi \left\|\frac{t}{2\pi} \right\|\leq 4\pi ^2 \left\|\frac{t}{2\pi} \right\|  \, .$$\\
Using \eqref{R_eta} we obtain \begin{equation}\label{minpr}
        \left|\sum_{q\in \mathcal{D}_Q}\sum_{\chi \in \mathcal{G}_q} \underset{\substack{|\mathfrak{I}(\rho_\chi)|\leq D_\eta, \\ \rho_\chi\notin \Gamma}} {{\sum}}  \M(\rho_\chi)e^{it_q\mathfrak{I}(\rho_\chi)}\right|>\frac{\eta(1)}{8}\sum_{q\in \mathcal{D}_Q}\Phi_q \log q \, .
    \end{equation}
Therefore, writing the explicit formula for each $\chi \in \mathcal{G}_q$ (Proposition \ref{explform}) we see that \begin{equation*}\label{Obj}
     \left|\sum_{q\in \mathcal{D}_Q}\sum_{\chi \in \mathcal{G}_q}\psi_\eta(e^{t_Q},\chi) \right|=\left | \sum_{q\in \mathcal{D}_Q}\sum_{\chi \in \mathcal{G}_q}\underset{\rho_\chi}{\sum}\M(\rho_\chi)e^{t_Q\rho_\chi} +O_\eta\left( \sum_{q\in \mathcal{D}_Q}\sum_{\chi \in \mathcal{G}_q}\log q\right) \right| 
     \end{equation*}
We then split each sum as follows $$ \underset{\rho_\chi}{\sum}\M(\rho_\chi)e^{t_Q\rho_\chi}=\underset{\substack{\frak{I}(\rho_\chi)\leq D_\eta, \\ \rho_\chi \notin \Gamma}}{{\sum}}\M(\rho_\chi)e^{t_Q\rho_\chi}+\sum_{\rho_\chi \in \Gamma}\M (\rho_\chi)e^{t_Q\rho_\chi}+\sum_{\frak{I}(\rho_\chi)>D_\eta}\M(\rho_\chi)e^{t_Q\rho_\chi}$$
We use $GRH_\eta$ to see that for all zeros $\rho_\chi$ we have $|\M (\rho_\chi)e^{t_Q\rho_\chi} |\leq |\M (\rho_\chi)|e^{\frac{t_Q}{2}}$.\\ Using the triangular inequality, we deduce \begin{align*}
    \left|\sum_{q\in \mathcal{D}_Q}\sum_{\chi \in \mathcal{G}_q}\psi_\eta(e^{t_Q},\chi) \right| &\geq \left |  \sum_{q\in \mathcal{D}_Q}\sum_{\chi \in \mathcal{G}_q} \underset{\substack{\frak{I}(\rho_\chi)\leq D_\eta,\\ \rho_\chi \notin \Gamma }}{{\sum}}\M(\rho_\chi)e^{it_q\mathfrak{I}(\rho_\chi)} \right|e^{\frac{t_Q}{2}}-  \sum_{q\in \mathcal{D}_Q}\sum_{\chi \in \mathcal{G}_q} \sum_{\rho_\chi \in \Gamma}\left| \M (\rho_\chi)  \right|e^{\frac{t_Q}{2}}\\
    &-  \sum_{q\in \mathcal{D}_Q}\sum_{\chi \in \mathcal{G}_q} \sum_{\frak{I}(\rho_\chi)>D_\eta}\left| \M(\rho_\chi)  \right|e^{\frac{t_Q}{2}}- O_\eta\left( \sum_{q\in \mathcal{D}_Q}\sum_{\chi \in \mathcal{G}_q}\log q\right)  \\
    &\gg_\eta \left(\sum_{q\in \mathcal{D}_Q}\Phi_q\right) e^{\frac{t_Q}{2}}\log Q \, ,
\end{align*}
where we used \eqref{Restmaj}, \eqref{minpr} and $\sum_{ \chi \in \mathcal{G}_q} \sum_{\rho_\chi \in \Gamma} \M (\rho_\chi) \ll_\eta q^\alpha $.\\
To end the proof, it suffices to take $x_Q:=e^{t_Q}$ and use \eqref{ Riemm-VM2} to see that $$\log \log x_Q = \log t_Q \asymp f(Q) \sum_{q\in \mathcal{D}_Q } N(D_\eta,\mathcal{G}_q)\asymp_\eta \sum_{q\in \mathcal{D}_Q } \Phi_q f(Q) \log Q \, .$$
\end{proof}
We are now ready to state and prove the main result of this section.
\begin{prop}\label{GRH-True}
     Let $\eta \in \mathcal{S}$. For all $Q\geq3$ let $\mathcal{D}_Q $ be a non-empty subset of $ [Q,2Q]\cap \N$, let $f:\N\to \R_{>C_\eta}$, where $C_\eta$ is defined in \eqref{C_eta}, and let $g:\N\to \R_{>0}$ such that uniformly for $Q\geq 3$ and $q\in \mathcal{D}_Q$ we have $$1\ll g(Q)\leq \log Q \text{ and } g(q) \asymp g(Q)\, .$$ If $GRH_\eta$ holds, then for all $Q$ large enough there exists $x_Q$ such that
     \begin{align}
         \log \log(x_Q) &\asymp_\eta \sum_{q\in \mathcal{D}_Q}\phi(q) g(Q) f(Q) ,\\
         \sum_{q\in \mathcal{D}_Q}V_\eta(x_Q;q) \gg_\eta& \frac{\sum_{q\in \mathcal{D}_Q}\phi(q)}{Q} x_Q g(Q) \log Q \, .
     \end{align}
     The same statement holds for $G_\eta(x_q; q)$ in place of $V_\eta(x_q;q)$.
\end{prop}
\begin{proof}
Let $B>0$ and $\mathcal{F}_q$ (for all $q>B$) as in Lemma \ref{subset}, and let $\mathcal{G}_q \subset \mathcal{F}_q$ of cardinal $$\Phi_q:=\left[\frac{g(q) |\mathcal{F}_q|}{\log q}\right]\asymp \frac{g(q) \phi(q)}{\log q}\, .$$
Thus, uniformly for $Q>B$ we have $$\sum_{q\in \mathcal{D}_Q} \Phi_q\asymp \frac{g(Q)}{\log Q}\sum_{q\in \mathcal{D}_Q} \phi(q)$$
By Lemma~\ref{subset} the sets $\mathcal{G}_Q$ for $Q>B$ satisfy the assumptions of Lemma~\ref{MainLemma}. For all $Q$ large enough there exists $x_Q$ such that 
\begin{align*}
         \log \log(x_Q) \asymp_\eta \sum_{q\in \mathcal{D}_Q} \Phi_q f(Q)& \log Q \asymp \sum_{q\in \mathcal{D}_Q} \phi(q) g(Q) f(Q) ,\\
         \left|\sum_{q\in \mathcal{D}_Q}\sum_{\chi \in \mathcal{G}_q}\psi_\eta(x_Q,\chi) \right| &\gg_\eta \sum_{q\in \mathcal{D}_Q}\Phi_q\sqrt{x_Q}\log Q \, .
     \end{align*}
By positivity and Cauchy-Schwarz inequality we have for all $Q$ large enough
\begin{align*}
    \sum_{q\in \mathcal{D}_Q} V_\eta (x_Q;q)&\geq \frac{1}{2Q}\sum_{q\in \mathcal{D}_Q}\sum_{\chi \in \mathcal{G}_q}|\psi_\eta(x_Q,\chi)|^2\\
                  &\geq \frac{1}{2Q\sum_{q\in \mathcal{D}_Q}\Phi_q}\left (\sum_{q\in \mathcal{D}_Q}\sum_{\chi \in \mathcal{G}_Q}|\psi_\eta(x_Q,\chi)|\right)^2\\
                  &\gg_\eta \frac{\sum_{q\in \mathcal{D}_Q}\Phi_q}{Q}(\log Q)^2x_Q \\
                  &\gg \frac{\sum_{q\in \mathcal{D}_Q}\phi(q)}{Q} x_Q g(Q) \log Q \, .
    \end{align*}
    To deduce a similar result for $G_\eta(x_q;q)$, we combine Lemmas~\ref{psi-theta}(2) and \ref{MainLemma} to see that for $Q$ large enough we have: $$\left|\sum_{q\in \mathcal{D}_Q}\sum_{\chi \in \mathcal{G}_q}\theta_\eta(x_Q,\chi) \right| \gg_\eta \left(\sum_{q\in \mathcal{D}_Q}\Phi_q\right) \sqrt{x_Q}\log Q \, ,$$
    then we conclude similarly using the Parseval identity~\eqref{Parseval}.
\end{proof}

\section{Proof of our main theorems}\label{proofs}

 \begin{proof}[Proof of Theorems~\ref{Hooley-False} and \ref{princ}]
It is sufficient to prove Theorem \ref{princ} which implies Theorem \ref{Hooley-False}.\\
Note that \eqref{property} implies that if $\log \log x_q \asymp \log \log h^{-1}(q)$ then $h(x_q) \asymp q$. Indeed if $\log \log x_q \leq V \log \log(h^{-1}(q))$, then $$ h(x_q)\leq h\left(\exp\left(\log(h^{-1}(q))^V\right)\right)\ll q\, ,$$
And similarly, if $\log \log (h^{-1}(q)) \leq W \log \log x_q$ then $$q\ll h(x_q)\, .$$

  If $GRH_\eta$ does not hold, then proposition \ref{GRH-False} implies the existence of a positive proportion of moduli $q$ for which we have $\log \log x_q \asymp_\eta \log \log h^{-1}(q)$ and for which $V_\eta(x_q;q)$ satisfies a stronger inequality than those stated in both cases of the Theorem.\\
 We may assume that $GRH_\eta$ holds.
    Let us prove case (1) of the statement, let $g:\N \to \R_{>0}$ be defined by $g(q)=\frac{\log \log(h^{-1}(q))}{q}$. Note that, up to multiplying $h$ by a large (absolute) positive constant, we may assume that $h(\exp(\exp(q \log q))\geq q$ since $$h(\exp(\exp(q \log q))\geq \frac{q \log q}{\log(q \log q)}\gg q \, .$$
    Hence $g(q) \leq \log q$. Note also that $\log \log(h^{-1}(q))\gg q$ hence $g(q) \gg 1$.  We apply Proposition \ref{GRH-True} with $\mathcal{D}_q=\{q\}$ for all $q\geq 3$ (and $q=Q$) and $f:\N \to \R$ given by $f(q)=C_\eta+1$ for all $q\in \N$, where $C_\eta$ is defined in \eqref{C_eta}, which implies that for large $q$, in terms of $\eta$, there exists $x_q$ such that $\log \log x_q \asymp_\eta \frac{\phi(q)}{q} \log \log h^{-1}(q)$ and $$V_\eta(x_q;q)\gg_{\eta} \frac{\phi(q)}{q} x_q g(q) \log(q)\,.$$
    Considering the positive proportion of moduli $q$ such that $\phi(q)>\frac{q}{2}$ we obtain the result.\\
    Let us prove case (2) of the statement, consider $g(q)=\log q$, $\mathcal{D}_q=\{q\}$ for $q\in \N$ and $f(q)=C_\eta \frac{\log \log(h^{-1}(q))}{q\log q}$ for $q\in \N$. As  $$h(\exp(\exp(q \log q)))\leq \frac{q \log q}{\log(q \log q)}< q\, .$$
    Thus, $\frac{\log \log(h^{-1}(q))}{q\log q}>1$. Hence, $f(q)>C_\eta$. We apply Proposition \ref{GRH-True} and we consider the positive proportion of moduli $q$ such that $\phi(q)>\frac{q}{2}$ this implies the result for $V_\eta$. The same proof is valid replacing $V_\eta$ by $G_\eta$.
\end{proof}

\begin{proof}[Proof of Theorems~\ref{Moyenne_simplifié} and \ref{moyenne}] We proceed as in the proof of Theorems \ref{princ} and \ref{Hooley-False}. It is enough to prove Theorem~\ref{moyenne}. Let us prove that $\log \log h^{-1}(2Q) \leq A_0 \log \log h^{-1}(Q)$. It suffices to prove that $$2Q\leq h \left(\exp\left( (\log h^{-1}(Q))^{A_0}\right)\right) \, ,$$ which is a direct consequence of \eqref{minoration_h}. Since $h^{-1}$ is increasing this implies that uniformly for $Q\geq 3$ and $Q<q\leq 2Q$ we have $$\log \log h^{-1}(q)\asymp \log \log h^{-1}(Q)\, .$$
If $GRH_\eta$ does not hold, then let $\delta>0$ and $q_0\geq 3$ given by Proposition \ref{GRH-False} and notice that $h(x)=o\left(x^{\varepsilon}\right)$ with $\varepsilon=\frac{\delta}{2}$. For all $Q\geq 3$ consider $q_Q$ the least multiple of $q_0$ such that $q_Q >Q$. For all $Q\geq q_0$ we have $Q< q_Q\leq 2Q$. By Proposition~\ref{GRH-False} for all $Q$ large enough there exists $x_{Q} \in \left[h^{-1}(q_Q);h^{-1}(q_Q)^{\frac{1}{1-\delta}}\right] \subset \left[h^{-1}(Q);(h^{-1}(2Q))^{\frac{1}{1-\delta}} \right]$ such that $V_\eta(x_Q;q_Q)\geq x_Q^{1+\delta}$. As $\frac{x_Q^\delta}{Q}\gg \frac{h(x_Q)^2}{h(x_Q)} = h(x_Q) \geq Q \gg (\log Q)^2$, by positivity we see that $$\frac{1}{Q} \sum_{Q<q\leq 2 Q} V_\eta(x_Q;q)\geq \frac{1}{Q} V_\eta(x_Q;q_Q) \gg_\eta x_Q (\log Q)^2 \, .$$ 
This implies both cases of the statement.\\
We may thus assume that $GRH_\eta$ holds. 
Let us prove case (1) of the statement. Let $g:\N \to \R_{>0}$ be defined by $g(Q)=\frac{\log \log(h^{-1}(Q))}{Q^2}$. Note that, up to multiplying $h$ by a large (absolute) positive constant, we may assume that $h(\exp(\exp(Q^2 \log Q))\geq Q$ since $$h(\exp(\exp(Q^2 \log Q))\geq \sqrt{\frac{Q^2 \log Q}{\log(Q^2 \log Q)}}\gg Q \, .$$
Hence $g(Q) \leq \log Q$. Note also that $\log \log(h^{-1}(Q))\gg Q^2$, hence $g(Q) \gg 1$. Since uniformly for $Q\geq 3$ and $Q<q\leq 2Q$ we have $\log \log h^{-1}(q)\asymp \log \log h^{-1}(Q)$ we have uniformly for such $Q$ and $q$ $$g(q)\asymp g(Q)\, .$$
We apply Proposition \ref{GRH-True} with $\mathcal{D}_Q=\{Q+1,Q+2,\cdots,2Q\}$ for all $Q\geq 3$  and $f:\N \to \R$ given by $f(Q)=C_\eta+1$ for all $Q\in \N$, which implies that for large $Q$, in terms of $\eta$, there exists $x_Q$ such that $$\log \log x_Q \asymp_\eta \frac{\sum_{Q< q \leq 2Q}\phi(q)}{Q^2} \log \log h^{-1}(Q)$$ and $$\frac{1}{Q}\sum_{Q<q\leq 2Q}V_\eta(x_Q;q)\gg_{\eta} \frac{\sum_{Q<q\leq 2Q}\phi(q)}{Q^2} x_Q g(Q) \log Q\, .$$
By the classical fact (see for instance \cite{Apostol}*{Theorem 3.7}): $$\sum_{n\leq x} \phi(n)=\frac{3}{\pi^2}x^2+O(x\log x)\, ,$$
we see that $$\sum_{Q<q\leq 2Q} \phi(q) \asymp Q^2 \, .$$ Thus, $g(Q)\asymp \frac{\log \log x_Q}{Q^2}$. Case (1) follows.\\
Let us prove case (2) of the statement. Consider $g(Q)=\log Q$, $\mathcal{D}_Q=\{Q+1,Q+2,\cdots, 2Q\}$ for $Q\in \N$ and $f:\N \to \R$ given by $f(Q):=C_\eta \frac{\log \log(h^{-1}(Q))}{Q^2\log Q}$ for $Q\in \N$. As  $$h(\exp(\exp(Q^2 \log Q))\leq \sqrt{\frac{Q^2 \log Q}{\log(Q^2 \log Q)}}< Q\, .$$
Thus, $\frac{\log \log(h^{-1}(Q))}{Q^2\log Q}>1$. Hence, $f(Q)>C_\eta$. By Proposition \ref{GRH-True} and since $\sum_{Q<q\leq 2Q} \phi(q) \asymp Q^2$ we obtain the result for $V_\eta$. The same proof is valid replacing $V_\eta$ by $G_\eta$. 
\end{proof}

\section*{Acknowledgements}

I would like to express my sincere gratitude to my supervisors, Daniel Fiorilli, who introduced me to this delightful topic and provided insightful comments and suggestions that greatly enhanced the quality of this research, and Florent Jouve, whose numerous discussions significantly contributed to the development of my ideas and who provided continuous encouragement.

\end{document}